\documentclass[leqno,12pt]{amsart}
\usepackage{amsmath,amstext,amssymb,amsopn,amsthm,mathrsfs}
\usepackage{epsfig,graphics,latexsym,graphicx,color}
\definecolor{darkgreen}{rgb}{0,.5,0}
\numberwithin{equation}{section}

\allowdisplaybreaks

\newtheorem{theor}{Theorem}[section]
\newtheorem{lemma}[theor]{Lemma}
\newtheorem{corollary}[theor]{Corollary}

\newtheorem{rem}{Remark}[section]

\usepackage{color}
\begin{document}
%\begin{CJK}{Bg5}{ttkai}
%%%%%%%%%%%%%%%%%%%%%%%%%%%%%%%%%%%%%%%%%%%%%%

\title[]{Characterizations of the BMO and Lipschitz spaces via commutators on weak Lebesgue and Morrey spaces }

%%%%%%%%%%%%%%%%%%%%%%%%%%%%%%%%%%%%%%%%%%%%%%
\author[]{Dinghuai Wang, Jiang Zhou$^\ast$ and Wenyi Chen}
\address{College of Mathematics and System Sciences \endgraf
         Xinjiang University \endgraf
         Urumqi 830046 \endgraf
         Republic of China}
\email{Wangdh1990@126.com; zhoujiangshuxue@126.com}
\address{School of Mathematics and Statistics \endgraf
         WuHan University \endgraf
         WuHan 430072 \endgraf
         Republic of China}
\email{wychencn@hotmail.com}
\thanks{The research was supported by National Natural Science Foundation
of China (Grant No.11261055). \\ \qquad * Corresponding author, Email: zhoujiangshuxue@126.com.}

%%%%%%%%%%%%%%%%%%%%%%%%%%%%%%%%%%%%%%%%%%%%%%
\begin{abstract}
We prove that the weak Morrey space $WM^{p}_{q}$ is contained in the Morrey space $M^{p}_{q_{1}}$ for $1\leq q_{1}< q\leq p<\infty$. As applications, we show that if the commutator $[b,T]$ is bounded from $L^p$ to $L^{p,\infty}$ for some $p\in (1,\infty)$, then $b\in \mathrm{BMO}$, where $T$ is a Calder\'on-Zygmund operator. Also, for $1<p\leq q<\infty$, $b\in \mathrm{BMO}$ if and only if $[b,T]$ is bounded from $M^{p}_{q}$ to $WM_{q}^{p}$. For $b$ belonging to Lipschitz class, we obtain similar results.
\end{abstract}
\maketitle

%%%%%%%%%%%%%%%%%%%%%%%%%%%%%%%%%%%%%%%%%%%%%%
%%%%%%%%%%%%%%%%%%%%%%%%%%%%%%%%%%%%%%%%%%%%%%
\maketitle

%%%%%%%%%%%%%%%%%%%%%%%%%%%%%%%%%%%%%%%%%%%%%%
%%%%%%%%%%%%%%%%%%%%%%%%%%%%%%%%%%%%%%%%%%%%%%
\section{Introduction} \label{sec:int}
Let $T$ be a Calder\'{o}n-Zygmund operator defined by
$$Tf(x)=\mathrm{p.v.}\int_{\mathbb{R}^{n}}K(x-y)f(y)dy,$$
where the kernel $K(x)=\frac{\Omega(x)}{|x|^{n}}$ satisfies the following conditions:
\begin{enumerate}
\item [\rm(i)] $\Omega$ is homogeneous of degree zero on $\mathbb{R}^{n}$, i.e., $\Omega(\lambda x)=\Omega(x)$ for all $\lambda>0$ and $x\in \mathbb{R}^{n}$;
\item [\rm(ii)] $\Omega\in C^{\infty}(\mathbb{S}^{n-1})$ and $\int_{\mathbb{S}^{n-1}}\Omega(x)dx=0.$
\end{enumerate}
A locally integrable function $b$ belongs to the $\mathrm{BMO}$ space if $b$ satisfies
$$\|b\|_{*}:=\sup_{Q}\frac{1}{|Q|}\int_{Q}|b(x)-b_{Q}|dx<\infty,$$
where $b_{Q}:=\frac{1}{|Q|}\int_{Q}b(x)dx$ and the supremum is taken over all cubes $Q$ in $\mathbb{R}^{n}$. A well known result of Coifman, Rochberg and Weiss \cite{CRW} states that the commutator
$$[b,T](f):=bT(f)-T(bf)$$
is bounded on some $L^p$, $1<p<\infty$, if and only if $b\in \mathrm{BMO}$. An interesting question is raised. Is $b$ in $\mathrm{BMO}$ if $[b,T]$ is of weak type $(p,p)$ for some $p\in (1,\infty)$? We will give an affirmative answer in this paper.

For $1\leq q\leq p <\infty$, we say that a function $f$ belongs to Morrey space $M^{p}_{q}$ if
$$\|f\|_{M^{p}_{q}}:=\sup_{x\in \mathbb{R}^{n},r>0}\frac{1}{|Q(x,r)|^{1/q-1/p}}\bigg(\int_{Q(x,r)}|f(y)|^{q}dy\bigg)^{1/q}<\infty;$$
a function $f$ belongs to weak Morrey space $WM_{q}^{p}$ if
$$\|f\|_{WM_{q}^{p}}:=\sup_{x\in \mathbb{R}^{n},r>0}\frac{1}{|Q(x,r)|^{1/q-1/p}}\sup_{\lambda>0}\bigg(\lambda^{q}|\{y\in Q(x,r):|f(y)|>\lambda\}|\bigg)^{1/q}<\infty.$$
Morrey spaces describe local regularity more precisely than $L^{p}$ spaces and can be seen as a complement of $L^{p}$. In fact, $L^p=M^{p}_{p}\subset M^{p}_{q}$ and $WM^{p}_{p}=L^{p,\infty}$ for $1\leq q \leq p<\infty$. In 1997, Ding \cite{D} showed that $b$ is in $\mathrm{BMO}$ if and only if the commutator $[b,T]$ of Calder\'on-Zygmund operator $T$ is bounded on Morrey spaces. We will demonstrate here that $b\in \mathrm{BMO}$ if and only if $[b,T]$ is weak bounded on Morrey spaces.

Another subject of this paper is to consider the characterizations of Lipschitz functions via commutators. For $0<\alpha<1$, the Lipschitz space $\mathrm{Lip}_{\alpha}$ is the set of functions $f$ such that
$$\|f\|_{\mathrm{Lip}_{\alpha}}:=\sup_{x,h\in \mathbb{R}^{n}\atop_{h\neq 0}}\frac{|f(x+h)-f(x)|}{|h|^{\alpha}}<\infty.$$
In 1978, Janson \cite{J} proved that, for $1<p<q<\infty$ with $1/q= 1/p-\alpha/n$, $b\in \mathrm{Lip}_{\alpha}$ if and only if $[b,T]$ is bounded from $L^p$ to $L^{q}$. In 1995, using Sobolev-Besov embedding, Paluszy\'nski \cite{PA} obtained that, for $1<p<\infty$ and $0<\alpha<1$, $b\in \mathrm{Lip}_{\alpha}$ if and only if $[b,T]$ is bounded from $L^p$ to the homogeneous Triebel-Lizorkin spaces $\dot{F}^{\alpha,\infty}_{p}$. Paluszy\'nski's idea was novel for the study about the boundedness of commutators from $L^p$ to $\dot{F}^{\alpha,\infty}_{p}$ and shed new light on the characterization of Lipschitz space via commutators. Recently, Shi and Lu \cite{SL} showed that, for $1<q\leq p<\infty$, $1<t\leq s<\infty$, $\frac{1}{s}=\frac{1}{p}-\frac{\alpha}{n}$ and $\frac{1}{t}=\frac{1}{q}-\frac{\alpha}{n}$, $b\in \mathrm{Lip}_{\alpha}$ if and only if $[b,T]$ is bounded from $M^{p}_{q}$ to $M^{s}_{t}$.
We are going to show that, for $1\leq p<q<\infty$ with $1/q=1/p-\alpha/n$, $b\in \mathrm{Lip}_{\alpha}$ if and only if $[b,T]$ is bounded from $L^p$ to $L^{q,\infty}$. Also, for $1\leq q\leq p<\infty$, $1<t\leq s<\infty$, $\frac{1}{s}=\frac{1}{p}-\frac{\alpha}{n}$ and $\frac{1}{t}=\frac{1}{q}-\frac{\alpha}{n}$, $b\in \mathrm{Lip}_{\alpha}$ if and only if $[b,T]$ is bounded from $M_{q}^{p}$ to $WM_{t}^{s}$.

Throughout this paper, the letter $C$ denotes constants which are independent of main variables and may change from one occurrence to another.

\section{Characterization of $\mathrm{BMO}$ space via commutators}

In this section, we characterize $\mathrm{BMO}$ space via the boundedness of commutator on (weak) Lebesgue spaces or (weak) Morrey spaces. First of all, we compare with Morrey spaces and weak Morrey spaces.

It is clear that $M^{p}_{q_{1}}$ is contained in $WM_{q_{2}}^{p}$ and $\|\cdot\|_{WM^{p}_{q_{2}}}\leq C\|\cdot\|_{M^{p}_{q_{1}}}$ if $1\leq q_{2}\leq q_{1}\leq p<\infty$.
However, for $1\leq q_{1}< q_{2}\leq p<\infty$, one has the reverse inequality as follows.
\begin{theor}\label{main}
If $1\leq q_{1}< q_{2}\leq p<\infty$, then $WM^{p}_{q_{2}}\subset M^{p}_{q_{1}}$ and $\|f\|_{M^{p}_{q_{1}}}\leq \break 2\big(\frac{q_{1}}{q_{2}-q_{1}}\big)^{\frac{1}{q_{2}}}\|f\|_{WM^{p}_{q_{2}}}$.
\end{theor}
\begin{proof}
Let $f\in WM^{p}_{q_{2}}$. Given a cube $Q \subset \mathbb{R}^n$ and $\lambda>0$,
$$\frac{1}{|Q|^{1/q_{2}-1/p}}\big(\lambda^{q_{2}}|\{x\in Q:|f(x)|>\lambda\}|\big)^{1/q_{2}}\leq \|f\|_{WM^{p}_{q_{2}}};$$
that is,
$$|\{x\in Q:|f(x)|>\lambda\}|\leq \|f\|^{q_{2}}_{WM^{p}_{q_{2}}}|Q|^{1-q_{2}/p}\lambda^{-q_{2}}.$$
Choose
$$N=\|f\|_{WM^{p}_{q_{2}}}|Q|^{-1/p}\Big(\frac{q_{1}}{q_{2}-q_{1}}\Big)^{1/q_{2}}.$$
Thus,
\begin{eqnarray*}
\int_{Q}|f(x)|^{q_{1}}dx&=&q_{1}\int_{0}^{\infty}\lambda^{q_{1}-1}|\{x\in Q:|f(x)|>\lambda\}|d\lambda\\
&\leq&q_{1}\int_{0}^{N}\lambda^{q_{1}-1}|Q|d\lambda+q_{1}\int_{N}^{\infty}\lambda^{q_{1}-1}\|f\|^{q_{2}}_{WM^{p}_{q_{2}}}|Q|^{1-q_{2}/p}\lambda^{-q_{2}}d\lambda\\
&=&|Q|N^{q_{1}}+\frac{q_{1}}{q_{2}-q_{1}}\|f\|^{q_{2}}_{WM^{p}_{q_{2}}}|Q|^{1-q_{2}/p}N^{q_{1}-q_{2}},
\end{eqnarray*}
which gives
$$\bigg(\int_{Q}|f(y)|^{q_{1}}dy\bigg)^{1/q_{1}}\leq 2\|f\|_{WM^{p}_{q_{2}}}|Q|^{1/q_{1}-1/p}\Big(\frac{q_{1}}{q_{2}-q_{1}}\Big)^{1/q}.$$
Then
$$\|f\|_{M^{p}_{q_{1}}}\leq 2\Big(\frac{q_{1}}{q_{2}-q_{1}}\Big)^{1/q_{2}}\|f\|_{WM^{p}_{q_{2}}}$$
and the lemma follows.
\end{proof}

\begin{rem}\label{rem}\rm
For $1\leq q_{1}<q_{2}< p<\infty$. Kozono and Yamazaki \cite[Lemma 1.7]{KY} proved that $\|f\|_{M_{q_{2}}^{p}}\leq C\|f\|_{L^{p,\infty}}$. The above Theorem \ref{main} yields immediately
$$\|f\|_{M^{p}_{q_{1}}}\leq C\|f\|_{WM_{q_{2}}^{p}}\leq C\|f\|_{M^{p}_{q_{2}}}\leq C\|f\|_{WM^{p}_{p}}=C\|f\|_{L^{p,\infty}}.$$
\end{rem}

Now we return to our first subject.

\begin{theor}\label{thm2.1}
Let $1< p<\infty$. The following statements are equivalent:
\begin{enumerate}
\item [\rm(1)] $b\in \mathrm{BMO}$;
\item [\rm(2)] $[b,T]$ is a bounded operator from $L^p$ to $L^p$;
\item [\rm(3)] $[b,T]$ is a bounded operator from $L^p$ to $L^{p,\infty}$.
\end{enumerate}
\end{theor}
\begin{theor}\label{thm2.2}
Let $1<q\leq p<\infty$. The following statements are equivalent:
\begin{enumerate}
\item [\rm(1)] $b\in \mathrm{BMO}$;
\item [\rm(2)] $[b,T]$ is a bounded operator from $M^{p}_{q}$ to $M^{p}_{q}$;
\item [\rm(3)] $[b,T]$ is a bounded operator from $M^{p}_{q}$ to $WM^{p}_{q}$.
\end{enumerate}
\end{theor}

\vskip 0.5cm
\noindent
{\it Proof of Theorem \ref{thm2.1}.} The equivalence of $(1)$ and $(2)$ was proved in [3]. By the inequality $\|f\|_{L^{p,\infty}}\leq \|f\|_{L^p}$, it is obvious that $(2)$ implies $(3)$.

To show $(3)\Rightarrow (1)$, we use Paluszy\'nski idea given in \cite{PA}. For $z_{0}\in \mathbb{R}^{n}\backslash \{0\}$, let $\delta=\frac{|z_{0}|}{2\sqrt{n}}$ and $Q(z_{0},\delta)$ denote the open cube centered at $z_{0}$ with side length $2\delta$. Then $K(x)^{-1}$ has an absolutely convergent Fourier series
$$\frac{1}{K(x)}=\sum a_{m}e^{i \langle v_{m},x\rangle}$$
with $\sum |a_{m}|<\infty$, where the exact form of the vectors $v_{m}$ is unrelated. Then, we have the expansion
$$\frac{1}{K(x)}=\frac{\delta^{-n}}{K(\delta x)}=\delta^{-n}\sum a_{m}e^{i\langle v_{m},\delta x\rangle} ~ for ~ |x-\frac{z_{0}}{\delta}|<\sqrt{n}.$$
Given cubes $Q=Q(x_{0},r)$ and $Q'=Q(x_{0}-rz_{1},r)$, if $x\in Q$ and $y\in Q'$, then
$$\Big|\frac{x-y}{r}-\frac{z_{0}}{\delta}\Big|\leq \Big|\frac{x-x_{0}}{r}\Big|+\Big|\frac{y-(x_{0}-\frac{r z_{0}}{\delta})}{r}\Big|< \sqrt{n}.$$
Let $s(x)=\overline{\mathrm{sgn}(\int_{Q'}(b(x)-b(y))dy)}$. Then
\begin{eqnarray*}
&&\frac{1}{|Q|}\int_{Q}|b(x)-b_{Q'}|dx\\
&&=\frac{1}{|Q|}\frac{1}{|Q'|}\int_{Q}\bigg|\int_{Q'}(b(x)-b(y))dy\bigg|dx\\
&&=\frac{1}{|Q|^{2}}\int_{Q}\int_{Q'}s(x)(b(x)-b(y))dydx\\
&&=\frac{1}{|Q|^{2}}\int_{Q}\int_{Q'}s(x)(b(x)-b(y))\frac{r^{n}K(x-y)}{K(\frac{x-y}{r})}dydx\\
&&=\frac{1}{|Q|}\int_{\mathbb{R}^{n}}\int_{\mathbb{R}^{n}}(b(x)-b(y))K(x-y)\sum a_{m}e^{i\langle v_{m}, \delta\frac{x-y}{r}\rangle}s(x)\chi_{Q}(x)\chi_{Q'}(y)dydx\\
&&=\frac{1}{|Q|}\sum a_{m}\int_{\mathbb{R}^{n}}\int_{\mathbb{R}^{n}}(b(x)-b(y))K(x-y)e^{i\langle v_{m},\frac{\delta}{r}x\rangle}
s(x)\chi_{Q}(x)e^{-i\langle v_{m},\frac{\delta}{r}y\rangle}\chi_{Q'}(y)dydx.
\end{eqnarray*}
Setting
$g_{m}(y)=e^{-i\langle v_{m},\frac{\delta}{r}y\rangle}\chi_{Q'}(y)$ and
$h_{m}(x)=e^{i\langle v_{m},\frac{\delta}{r}x\rangle}s(x)\chi_{Q}(x)$, we have
\begin{eqnarray*}
\frac{1}{|Q|}\int_{Q}|b(x)-b_{Q'}|dx&=&\frac{1}{|Q|}\sum a_{m}\int_{\mathbb{R}^{n}}\int_{\mathbb{R}^{n}}(b(x)-b(y))K(x-y)g_{m}(y)h_{m}(x)dydx\\
&=&\frac{1}{|Q|}\sum a_{m}\int_{\mathbb{R}^{n}}[b,T](g_{m})(x)h_{m}(x)dx.
\end{eqnarray*}
Choose $q\in (1,p)$. By Remark \ref{rem},
\begin{eqnarray}\label{eq1}
\begin{split}
\frac{1}{|Q|}\int_{Q}|b(x)-b_{Q}|dx&\leq \frac{2}{|Q|}\int_{Q}|b(x)-b_{Q'}|dx\\
&\leq\frac{C}{|Q|^{1/p}}\sum |a_{m}|\|[b,T](g_{m})\|_{M^{p}_{q}}\\
&\leq\frac{C}{|Q|^{1/p}}\sum |a_{m}|\|[b,T](g_{m})\|_{L^{p,\infty}}\\
&\leq\frac{C\|g_{m}\|_{L^p}}{|Q|^{1/p}}\|[b,T]\|_{L^p\rightarrow L^{p,\infty}}\\
&\leq C\|[b,T]\|_{L^p\rightarrow L^{p,\infty}},
\end{split}
\end{eqnarray}
which yields $b\in \mathrm{BMO}$ and $\|b\|_{*}\leq C\|[b,T]\|_{L^p\rightarrow L^{p,\infty}}$. Hence, the proof of Theorem \ref{thm2.1} is completed. \qed

\vskip 0.5cm

\noindent
{\it Proof of Theorem \ref{thm2.2}.} We use a same argument as the proof of Theorem \ref{thm2.1} except choosing $q_{1}\in(1,q)$ and replacing \eqref{eq1} by
\begin{eqnarray*}
\frac{1}{|Q|}\int_{Q}|b(x)-b_{Q}|dx&\leq& \frac{2}{|Q|}\int_{Q}|b(x)-b_{Q'}|dx\\
&\leq&\frac{C}{|Q|^{1/p}}\sum |a_{m}|\|[b,T](g_{m})\|_{M^{p}_{q_{1}}}\\
&\leq&\frac{C}{|Q|^{1/p}}\sum |a_{m}|\|[b,T](g_{m})\|_{WM^{p}_{q}}\\
&\leq&\frac{C\|g_{m}\|_{M^p_q}}{|Q|^{1/p}}\|[b,T]\|_{M^p_q\rightarrow WM^{p}_{q}}\\
&\leq&C\|[b,T]\|_{M^p_q\rightarrow WM^{p}_{q}}.
\end{eqnarray*}
Then $b\in \mathrm{BMO}$ and $\|b\|_{*}\leq C\|[b,T]\|_{M^p_q\rightarrow WM^{p}_{q}}$. This completes the proof of Theorem \ref{thm2.2}. \qed

\section{Characterization of $\mathrm{Lip}_{\alpha}$ via commutators}

We give a lemma that can be used to prove a characterization of Lipschitz functions.

\begin{lemma}\label{lem3.1}
For $0<\alpha<1$ and $1<q\leq \infty$,
\begin{align*}
\|b\|_{\mathrm{Lip}_{\alpha}}&\approx \sup_{Q}\frac{1}{|Q|^{1+\alpha/n}}\int_{Q}|b(x)-b_{Q}|dx\\
&\approx
\begin{cases}\displaystyle
\sup_{Q}\frac{1}{|Q|^{\alpha/n}}\bigg(\frac{1}{|Q|}\int_{Q}|b(x)-b_{Q}|^{q}dx\bigg)^{1/q} &\text{if}\ \ 1<q<\infty\\
\displaystyle\mathop\mathrm{ess~sup}_{Q}\frac{|b(x)-b_{Q}|}{|Q|^{\alpha/n}} &\text{if}\ \ q=\infty
\end{cases},
\end{align*}
where the supremum is taken over all cubes $Q\subset\mathbb{R}^{n}$ and $\approx$ means equivalence.
\end{lemma}
\begin{proof}
The first equivalence can be found in \cite[pages 14 and 38]{DS}, and the second equivalence can be found in \cite{JTW}.
\end{proof}

The first result of this section is

\begin{theor}\label{thm3.1}
Let $0<\alpha<1$, $1\leq p<\frac{n}{\alpha}$ and $1/q=1/p-\alpha/n$. The following statements are equivalent:
\begin{enumerate}
\item [\rm(1)] $b\in \mathrm{Lip}_{\alpha}$;
\item [\rm(2)] $[b,T]$ is a bounded operator from $L^p$ to $L^{q,\infty}$.
\end{enumerate}
\end{theor}

\begin{proof}
$(1)\Rightarrow (2)$: Let $b\in \mathrm{Lip}_{\alpha}$. Then
\begin{eqnarray}\label{eq2}
\begin{split}
\big|[b,T](f)(x)\big|&=\bigg|\int_{\mathbb{R}^{n}}(b(x)-b(y))f(y)K(x-y)dy\bigg|\\
&\leq\|b\|_{\mathrm{Lip}_{\alpha}}\int_{\mathbb{R}^{n}}\frac{|f(y)|}{|x-y|^{n-\alpha}}dy\\
&\leq\|b\|_{\mathrm{Lip}_{\alpha}}I_{\alpha}(|f|)(x),
\end{split}
\end{eqnarray}
which implies
$$\|[b,T]f\|_{L^{q,\infty}}\leq \|b\|_{\mathrm{Lip}_{\alpha}}\|I_{\alpha}(|f|)\|_{L^{q,\infty}}\leq C\|b\|_{\mathrm{Lip}_{\alpha}}\|f\|_{L^{p}}.$$

$(2)\Rightarrow (1)$: We follow the method in the proof of Theorem \ref{thm2.1} except choosing $q_{1}\in(1,q)$ and replacing \eqref{eq1} by
\begin{eqnarray*}
\frac{1}{|Q|}\int_{Q}|b(x)-b_{Q}|dx&\leq& \frac{2}{|Q|}\int_{Q}|b(x)-b_{Q'}|dx\\
&\leq&\frac{C}{|Q|^{1/q}}\sum |a_{m}|\|[b,T](g_{m})\|_{M^{q}_{q_{1}}}\\
&\leq&\frac{C}{|Q|^{1/q}}\sum |a_{m}|\|[b,T](g_{m})\|_{L^{q,\infty}}\\
&\leq&\frac{C\|g_{m}\|_{L^p}}{|Q|^{1/q}}\|[b,T]\|_{L^p\rightarrow L^{q,\infty}}\\
&\leq&C|Q|^{\alpha/n}\|[b,T]\|_{L^p\rightarrow L^{q,\infty}}.
\end{eqnarray*}
From Lemma \ref{lem3.1} we conclude that $b\in \mathrm{Lip}_{\alpha}$ and $\|b\|_{\mathrm{Lip}_{\alpha}}\leq C\|[b,T]\|_{L^p\rightarrow L^{q,\infty}}$. Hence, the proof of Theorem \ref{thm3.1} is completed.
\end{proof}

Using Theorem \ref{thm3.1} and Janson's result \cite{J}, we immediately have
\begin{corollary}
Let $0<\alpha<1$, $1< p<\frac{n}{\alpha}$ and $1/q=1/p-\alpha/n$. The following statements are equivalent:
\begin{enumerate}
\item  $b\in \mathrm{Lip}_{\alpha}$;
\item  $[b,T]$ is a bounded operator from $L^p$ to $L^{q}$;
\item  $[b,T]$ is a bounded operator from $L^p$ to $L^{q,\infty}$.
\end{enumerate}
\end{corollary}

We give two remarkable results about the boundedness of fractional integral operator $I_{\alpha}$ on Morrey spaces, where $I_{\alpha}$ is defined by
$$I_{\alpha}f(x)=\int_{\mathbb{R}^n}\frac{f(y)}{|x-y|^{n-\alpha}}dy,\qquad x\in\mathbb{R}^{n} \ and \ 0<\alpha<n.$$

\begin{theor}\label{P}{\rm(Peetre \cite{PE})}
Let $0<\alpha<n$, $1<q\leq p<\infty$ and $1<t\leq s<\infty$. If $\frac{1}{s}=\frac{1}{p}-\frac{\alpha}{n}$ and $\frac{1}{t}=\frac{1}{q}-\frac{\alpha}{n}$, then $\|I_{\alpha}f\|_{M^{s}_{t}}\leq C\|f\|_{M^{p}_{q}}$.
\end{theor}

\begin{theor}\label{A}{\rm(Adams \cite{A})}
Let $0<\alpha<n$, $1<q\leq p<\infty$ and $1<l\leq s<\infty$. If $\frac{1}{s}=\frac{1}{p}-\frac{\alpha}{n}$ and $\frac{l}{s}=\frac{q}{p}$, then  $\|I_{\alpha}f\|_{M^{s}_{l}}\leq C\|f\|_{M^{p}_{q}}$.
\end{theor}

We will use the above two theorems to show other characterizations of Lipschitz functions.

\begin{theor}\label{thm3.2}
Let $0<\alpha<1$, $1\leq q\leq p<\infty$, $1<t\leq s<\infty$, $\frac{1}{s}=\frac{1}{p}-\frac{\alpha}{n}$ and $\frac{1}{t}=\frac{1}{q}-\frac{\alpha}{n}$. The following statements are equivalent:
\begin{enumerate}
\item  $b\in \mathrm{Lip}_{\alpha}$;
\item  $[b,T]$ is a bounded operator from $M^{p}_{q}$ to $WM^{s}_{t}$.
\end{enumerate}
\end{theor}

\begin{theor}\label{thm3.3}
Let $0<\alpha<1$, $1< q\leq p<\infty$, $1<t\leq s<\infty$, $1<l\leq s<\infty$, $\frac{1}{s}=\frac{1}{p}-\frac{\alpha}{n}$, $\frac{1}{t}=\frac{1}{q}-\frac{\alpha}{n}$ and $\frac{l}{s}=\frac{q}{p}$,. The following statements are equivalent:
\begin{enumerate}
\item  $b\in \mathrm{Lip}_{\alpha}$;
\item  $[b,T]$ is a bounded operator from $M^{p}_{q}$ to $M^{s}_{l}$;
\item  $[b,T]$ is a bounded operator from $M^{p}_{q}$ to $M^{s}_{t}$;
\item  $[b,T]$ is a bounded operator from $M^{p}_{q}$ to $WM^{s}_{t}$.
\end{enumerate}
\end{theor}

\noindent
{\it Proof of Theorem \ref{thm3.2}.} $(1)\Rightarrow (2)$: For $q>1$, Theorem \ref{P} and \eqref{eq2} yield
$$\|[b,T]f\|_{WM^{s}_{t}}\leq \|b\|_{\mathrm{Lip}_{\alpha}}\|I_{\alpha}(|f|)\|_{WM_{t}^{s}}\leq \|b\|_{\mathrm{Lip}_{\alpha}}\|I_{\alpha}(|f|)\|_{M_{t}^{s}}\leq C\|b\|_{\mathrm{Lip}_{\alpha}}\|f\|_{M^{p}_{q}}.$$
For $q=1$, it follows from \cite[Theorem 1.4]{T} that
$$\|I_{\alpha}(f)\|_{WM_{1}^{s}}\leq C\|f\|_{M^{p}_{q}},$$
which gives
$$\|[b,T]f\|_{WM^{s}_{1}}\leq \|b\|_{\mathrm{Lip}_{\alpha}}\|I_{\alpha}(|f|)\|_{WM_{1}^{s}}\leq C\|b\|_{Lip_{\alpha}}\|f\|_{M^{p}_{q}}.$$

$(2)\Rightarrow (1)$: A same argument as the proof of Theorem \ref{thm2.1} except choosing $q_{1}\in(1,t)$ and replacing \eqref{eq1} by
\begin{eqnarray*}
\frac{1}{|Q|}\int_{Q}|b(x)-b_{Q}|dx&\leq& \frac{2}{|Q|}\int_{Q}|b(x)-b_{Q'}|dx\\
&\leq&\frac{C}{|Q|^{1/s}}\sum |a_{m}|\|[b,T](g_{m})\|_{M^{s}_{q_{1}}}\\
&\leq&\frac{C}{|Q|^{1/s}}\sum |a_{m}|\|[b,T](g_{m})\|_{WM^{s}_{t}}\\
&\leq&C|Q|^{\alpha/n}\|[b,T]\|_{M^p_q\rightarrow WM^s_t}.
\end{eqnarray*}
implies $b\in \mathrm{Lip}_{\alpha}$ due to Lemma \ref{lem3.1}. \qed

\vskip 0.5cm
\noindent
{\it Proof of Theorem \ref{thm3.3}.} $(1)\Rightarrow (2)$: It follows from \eqref{eq2} and Theorem \ref{A} that $[b,T]$ is bounded from $M^{p}_{q}$ to $M^{s}_{l}$.

$(2)\Rightarrow (3)$: Since $\frac{1}{s}=\frac{1}{p}-\frac{\alpha}{n}$, $\frac{1}{t}=\frac{1}{q}-\frac{\alpha}{n}$ and $\frac{l}{s}=\frac{q}{p}$, we have
$$\frac{1}{q}-\frac{1}{t}=\frac{1}{p}-\frac{1}{s}=\frac{q}{p}\Big(\frac{1}{q}-\frac{1}{l}\Big)\leq \frac{1}{q}-\frac{1}{l},$$
and hence $t\leq l$. Then $ \|\cdot\|_{M^{s}_{t}}\leq\|\cdot\|_{M^{s}_{l}}$.

It is obvious for $(3)\Rightarrow (4)$.

$(4)\Rightarrow (1)$: By Theorem \ref{thm3.2}, we obtain $b\in \rm{Lip}_{\alpha}$.
Hence, the proof of Theorem \ref{thm3.3} is completed.
$\hfill$\qed

%%%%%%%%%%%%%%%%%%%%%%%%%%%%%%%%%%%%%%%%%%%%%%%%%%%%%%%%%%%%%%%%%%%%%%
%%%%%%%%%%%%%%%%%%%%%%%%%%%%%%%%%%%%%%%%%%%%%%%%%%%%%%%%%%%%%%%%%%%%%%
\color{black}

%\end{CJK}
\end{document}